\documentclass[11pt,reqno]{amsart}
\usepackage{amsthm,amssymb,amsmath}
\usepackage{xcolor}

\newtheorem{theorem}{Theorem}
\newtheorem*{theorem*}{Theorem}

\newtheorem{corollary}{Corollary}
\newtheorem{proposition}{Proposition}

\newtheorem*{theoremA}{Theorem A}

\theoremstyle{definition}

\newtheorem*{definition*}{\sc Definition}

\newtheorem{remark}{Remark}

\newtheorem*{example*}{\bf Example}

\newcommand{\loc}{{\rm loc}}

\newcommand{\Real}{{\rm Re}}

\newcommand{\sprt}{{\rm sprt\,}}
\newcommand{\clos}{{\rm clos}}

\makeatletter
\newcommand{\dotminus}{\mathbin{\text{\@dotminus}}}

\newcommand{\@dotminus}{%
  \ooalign{\hidewidth\raise1ex\hbox{.}\hidewidth\cr$\m@th-$\cr}%
}
\makeatother

\begin{document}

\title[Fundamental solution to fractional Schr\"odinger operator]{Two-sided weighted bounds on fundamental solution to fractional Schr\"odinger operator}

\author{D.\,Kinzebulatov}

\address{Universit\'{e} Laval, D\'{e}partement de math\'{e}matiques et de statistique, 1045 av.\,de la M\'{e}decine, Qu\'{e}bec, QC, G1V 0A6, Canada}

\email{damir.kinzebulatov@mat.ulaval.ca}

\author{Yu.\,A.\,Sem\"{e}nov}

\address{University of Toronto, Department of Mathematics, 40 St.\,George Str, Toronto, ON, M5S 2E4, Canada}

\email{semenov.yu.a@gmail.com}



\begin{abstract}
We establish sharp two-sided weighted bounds on the fundamental solution to the fractional Schr\"{o}dinger operator using the method of desingularizing weights.
\end{abstract}

\maketitle
 
In \cite{MS0}, Milman and Sem\"{e}nov developed an approach to study of the integral kernels of semigroups which are not necessarily ultracontractive by transferring them to appropriately chosen weighted spaces
where they become ultracontractive \cite{MS1,MS}. In the special case of the Schr\"{o}dinger semigroup generated by $-\Delta - V$, with potential $V(x)=\delta\frac{(d-2)^2}{4}|x|^{-2}$, $0<\delta \leq 1$, $d \geq 3$, having a critical-order singularity at $x=0$ (which makes invalid the standard two-sided Gaussian bounds on its integral kernel) this method yields sharp two-sided \textit{weighted} bounds on the integral kernel.

In \cite{KiS2}, we employed the method of desingularizing weights to establish sharp two-sided weighted bounds on the fundamental solution to the non-local operator
$$
(-\Delta)^{\frac{\alpha}{2}} + b \cdot \nabla, \quad b(x)=\delta(d-\alpha)^{-2} 2 c^{-2}_\alpha|x|^{-\alpha}x, \quad 0<\delta<1, \quad 1<\alpha<2, \quad d \geq 3,
$$
where $c_\alpha:=c\left(\frac{\alpha}{2},2,d\right)$,  
\[
c(\alpha,p,d):=\frac{\gamma(\frac{d}{p}-\alpha)}{\gamma(\frac{d}{p})}, \quad \gamma(\alpha):=\frac{2^\alpha\pi^\frac{d}{2}\Gamma(\frac{\alpha}{2})}{\Gamma(\frac{d}{2}-\frac{\alpha}{2})}, \quad 1<p<\frac{d}{\alpha}.
\]
In this paper, we specify our arguments in \cite{KiS2} to the operator
$$
(-\Delta)^{\frac{\alpha}{2}} - V, \quad
V(x)=\delta c_\alpha^{-2}|x|^{-\alpha},\quad 0<\delta \leq 1, \quad 0<\alpha<2,$$
and obtain sharp two-sided weighted bounds on its fundamental solution.
These bounds are known, see \cite{BGJP}, where the authors use a different technique.
Concerning $(-\Delta)^{\frac{\alpha}{2}} + c|x|^{-\alpha}$, $c>0$, see \cite{CKSV} and \cite{JW}.

\medskip

\textbf{1.~}The method of desingularizing weights relies on two assumptions: the Sobolev embedding property, and a ``desingularizing''
$(L^1,L^1)$ bound on the weighted semigroup.
Let $X$ be a locally compact topological space and $\mu$ a $\sigma$-finite Borel measure on $X$. Let $\Lambda$ be a non-negative selfadjoint operator in the (complex) Hilbert space $L^2=L^2(X,\mu)$ with the inner product $\langle f,g\rangle =\int_X f\bar{g} d \mu$. We assume that $\Lambda$ possesses the Sobolev embedding property:

There are constants $j>1$ and $c_S>0$ such that, for all $f\in D(\Lambda^\frac{1}{2})$,
\[
c_S \|f \|^2_{2j} \leq \|\Lambda^\frac{1}{2}f\|_2^2 \tag{$M_1$}
\]
but $e^{-t\Lambda}\upharpoonright L^1\cap L^2$, $t>0$, cannot be extended by continuity to a bounded map on $L^1$ and the ultracontractive estimate
\[
\|e^{-t\Lambda}f\|_\infty \leq c_t \|f\|_1, \;\; f\in L^1\cap L^\infty, \; t>0
\]
is not valid.

In this case we will be assuming that there exists a family of real valued weights $\varphi=\{\varphi_s\}_{s>0}$ on $X$ such that, for all $s>0$,
\[
\varphi_s, \; 1/\varphi_s \in L^2_\loc(X,\mu) \tag{$M_2$}
\] 
and there exists a constant $c_1$ independent of $s$ such that, for all $0<t\leq s$,
\[
\|\varphi_s e^{-t\Lambda}\varphi_s^{-1} f\|_1 \leq c_1 \|f\|_1, \;\; f\in \mathcal
D:=\varphi_s L^\infty_{com} (X,\mu). \tag{$M_3$}
\]

\begin{theoremA}[\cite{MS}]
\label{weightM1}

In addition to $(M_1)-(M_3)$ assume that
\[
\inf_{s>0, x\in X} |\varphi_s(x)| \geq c_0 >0. \tag{$M_4$}
\]
Then $e^{-tA}, t>0$ are integral operators, and there is a constant $C=C(j,c_s,c_1,c_0)$ such that, for all $t>0$ and $\mu$ a.e. $x,y \in X$,
\[
|e^{-t\Lambda}(x,y)|\leq Ct^{-j^\prime} |\varphi_t(x) \varphi_t(y)|, \;\; j^\prime=j/(j-1). 
\label{nie}
\tag{$NIE_w$}
\]
\end{theoremA}

For the sake of completeness, we recall the proof of Theorem A in Appendix \ref{A}.

In applications of Theorem A to concrete operators the main difficulty consists in
verification of the $(L^1,L^1)$ bound ($M_3$). 
In \cite{MS}, ($M_3$) is proved  for the Schr\"{o}dinger operator
by means of the theory of $m$-sectorial operators and the Stampacchia criterion in $L^2$. 
However, attempts to apply that argument to $(-\Delta)^{\frac{\alpha}{2}}$, $\alpha<2$,  are quite problematic since $(-\Delta)^{\frac{\alpha}{2}}$ lacks the local properties of $-\Delta$. 
In \cite{KiS2}, we developed a new approach to the proof of ($M_3$) by means of the Lumer-Phillips Theorem applied to specially constructed $C_0$ semigroups in $L^1$ which approximate $\varphi_s e^{-t\Lambda}\varphi_s^{-1}$. 
Thus, in contrast to \cite{MS}, where the $(L^1,L^1)$ bound is proved using the $L^2$ theory, here we stay within the $L^1$ theory.
For $\alpha=2$, the approximation semigroups are constructed by replacing $|x|$ by $|x|_\varepsilon=\sqrt{|x|^2+\varepsilon}$, $\varepsilon>0$, both in the potential and in the weights, see below. For $\alpha<2$, the construction of the approximation semigroups is more subtle, and is a key observation.

\medskip

\textbf{2.~}We now state our result on $(-\Delta)^{\frac{\alpha}{2}}-V$ in detail. First, we consider the case $0<\delta<1$. The case $\delta=1$ requires few modifications. We will attend to it in the end.

According to the Hardy-Rellich inequality 
$\| (-\Delta)^{\frac{\alpha}{4}}f\|_2^2 \geq c^{-2}(\frac{\alpha}{2},2,d) \||x|^{-\frac{\alpha}{2}}f\|_2^2$ (see \cite[Lemma 2.7]{KPS}) 
the form difference $\Lambda=(-\Delta)^\frac{\alpha}{2}\dotminus V$ is well defined \cite[Ch.VI, sect 2.5]{Ka}.

Define $\beta$ by $\delta c_\alpha^{-2}=\frac{\gamma(\beta)}{\gamma(\beta-\alpha)}$, and let $\varphi(x)\equiv \varphi_s(x)=\eta(s^{-\frac{1}{\alpha}}|x|)$, where $\eta \in C^2(\mathbb R - \{0\})$ is such that
\[
\eta(r)=\left \{
\begin{array}{ll}
r^{-d+\beta}, & 0<r< 1, \\
\frac{1}{2}, & r\geq 2.
\end{array}
\right.
\]

\begin{theorem}
\label{thm1} 
Under constraints $0<\delta<1$ and $0<\alpha<2$, $e^{-t\Lambda}$ is an integral operator for each $t>0$. The weighted Nash initial estimate 
\[
e^{-t\Lambda}(x,y)\leq c t^{-\frac{d}{\alpha}}\varphi_t(x)\varphi_t(y),\quad c=c_{d,\delta,\alpha},
\]
is valid for all $t>0$, $x,y\in \mathbb R^d-\{0\}$.
\end{theorem} 

\begin{proof}[Proof of Theorem \ref{thm1}]
We verify the assumptions of Theorem A:

($M_1$) follows from the Hardy-Rellich inequality and the uniform Sobolev inequality $\| (-\Delta)^{\frac{\alpha}{4}}f\|^2_2 \geq c_S\|f\|_{2j}^2$, $j=\frac{d}{d-\alpha}$.

($M_2$), ($M_4$) are immediate from the definition of $\varphi_s$.

($M_3$) Our goal is to prove the following $(L^1,L^1)$ bound:
\[
\|\varphi e^{-t\Lambda}\varphi^{-1}h\|_1\leq e^{c\frac{t}{s}}\|h\|_1, \quad h\in L^1\cap L^2, \;\;t>0. \tag{$\bullet$}
\label{M3}
\]

\textit{Proof of \eqref{M3}.~}In $L^1$ define operator $\Lambda^\varepsilon= (-\Delta)^\frac{\alpha}{2}-V_\varepsilon$, $V_\varepsilon(x)=\delta c_\alpha^{-2}|x|_\varepsilon^{-\alpha}$, $\varepsilon>0$, $D(\Lambda^\varepsilon)=D((-\Delta)^\frac{\alpha}{2})$, 
\[
Q=\phi_n \Lambda^\varepsilon \phi_n^{-1}, \quad D(Q)=\phi_n D(\Lambda^\varepsilon), \quad F_{\varepsilon,n}^t=\phi_n e^{-t\Lambda^\varepsilon}\phi_n^{-1},\quad \phi_n(x)=e^{-\frac{\Lambda^\varepsilon}{n}}\varphi(x),\quad n=1,2,\dots
\]
Here $\phi_n D(\Lambda^\varepsilon):=\{\phi_n u\mid u\in D(\Lambda^\varepsilon)\}$. We also note that $e^{-t(-\Delta)^\frac{\alpha}{2}}$, $e^{-t\Lambda^\varepsilon}:\mathcal M\to \mathcal M$ where $\mathcal M=C_u$ or $\mathcal M=L^1$; and $\varphi=\varphi^{(1)}+\varphi^{(u)}$, $\varphi^{(1)}\in D((-\Delta)^\frac{\alpha}{2}_{L^1})$, $\varphi^{(u)}\in D((-\Delta)^\frac{\alpha}{2}_{C_u})$. $C_u\equiv C_u(\mathbb R^d)$ stands for the Banach space of uniformly continuous functions endowed with the supremum norm.

Since $\phi_n, \phi_n^{-1}\in L^\infty$, the operators $Q$, $F_{\varepsilon,n}^t$ are well defined.
 
1.~Clearly, $F^t_{\varepsilon,n}$ is a quasi bounded $C_0$ semigroup in $L^1$, say $e^{-tG}$. Set 
\[
\quad M:= \phi_n(1+(-\Delta)^\frac{\alpha}{2})^{-1}[L^1\cap C_u]=\phi_n(\lambda_\varepsilon+\Lambda^\varepsilon)^{-1}[L^1\cap C_u], \;\; 0<\lambda_\varepsilon\in \rho(-\Lambda^\varepsilon).
\]
Clearly, $M$ is a dense subspace of $L^1$, $M\subset D(Q)$ and $M\subset D(G)$. Moreover, $Q\upharpoonright M\subset G$. Indeed, for $f=\phi_n u\in M$,
\[
Gf=s\mbox{-}L^1\mbox{-}\lim_{t\downarrow 0}t^{-1}(1-e^{-tG})f=\phi_n s\mbox{-}L^1\mbox{-}\lim_{t\downarrow 0}t^{-1}(1-e^{-t\Lambda^\varepsilon})u=\phi_n\Lambda^\varepsilon u=Qf.
\]
Thus $Q\upharpoonright M$ is closable and $\tilde{Q}:=(Q\upharpoonright M)^{\clos}\subset G$.

Next, let us show that $R(\lambda_\varepsilon +\tilde{Q})$ is dense in $L^1$.
If $\langle(\lambda_\varepsilon +\tilde{Q})h,v\rangle=0$ for all $h\in D(\tilde{Q})$ and some $v\in L^\infty$, $\|v\|_\infty=1$, then taking $h\in M$ we would have $\langle(\lambda_\varepsilon+Q)\phi_n (\lambda_\varepsilon+\Lambda^\varepsilon)^{-1}g,v\rangle=0$, $g\in L^1\cap C_u$, or $\langle\phi_n g,v\rangle=0$. Choosing $g=e^\frac{\Delta}{k}(\chi_m v)$, where $\chi_m\in C^\infty_c$ with $\chi_m(x)=1$ when $x\in B(0,m)$, we would have $\lim_{k\uparrow\infty}\langle \phi_n g,v\rangle=\langle\phi_n\chi_m,|v|^2\rangle=0$, and so $v\equiv 0$. Thus, $R(\lambda_\varepsilon +\tilde{Q})$ is dense in $L^1$. 

2.~The main step:

\begin{proposition}
\label{prop_main}
There is a constant $\hat{c}=\hat{c}(d,\alpha,\delta)$ such that
 \[
  \lambda+\tilde{Q} \textit{ is accretive whenever } \lambda\geq \hat{c} s^{-1}.
 \] 
\end{proposition}

Taking Proposition \ref{prop_main} for granted we immediately establish the bound
 \[
 \|e^{-tG}\|_{1\to 1}\equiv\|\phi_ne^{-t\Lambda^\varepsilon}\phi_n^{-1}\|_{1\to 1}\leq e^{\omega t} , \quad \omega=\hat{c} s^{-1}. \tag{$\star$}
 \]
Indeed, the facts: $\tilde{Q}$ is closed and $R(\lambda_\varepsilon +\tilde{Q})$ is dense in $L^1$ \textit{together with} Proposition \ref{prop_main} imply $R(\lambda_\varepsilon +\tilde{Q})=L^1$. But then, by the Lumer-Phillips Theorem, $\lambda+\tilde{Q}$ is the (minus) generator of a contraction $C_0$ semigroup, and $\tilde{Q}=G$ due to $\tilde{Q}\subset G$. Incidentally, $M$ is a core of $G$.
 
In turn,  $(\star)$ easily yields
\[
 \|\varphi e^{-t\Lambda^\varepsilon}\varphi^{-1}h\|_1\leq e^{\omega t}\|h\|_1, \quad h\in L^1\cap L^2. \tag{$\star\star$}
\]
Indeed, $(\star)$ implies that $\lim_{n\uparrow \infty}\|\phi_n e^{-t\Lambda^\varepsilon}v\|_1\leq e^{\omega t}\lim_{n\uparrow \infty}\|\phi_n v\|_1$ for all $v\in L^1\cap L^2$. But
\[
\lim_{n\uparrow \infty}\|\phi_n v\|_1= \lim_{n\uparrow \infty}\langle \varphi,e^{-\frac{\Lambda^\varepsilon}{n}}|v|\rangle=\langle \varphi,|v|\rangle<\infty,
\]
\[
\lim_{n\uparrow \infty}\|\phi_n e^{-t\Lambda^\varepsilon}v\|_1 = \lim_{n\uparrow \infty}\langle \varphi,e^{-\frac{\Lambda^\varepsilon}{n}}|e^{-t\Lambda^\varepsilon} v|\rangle=\langle \varphi,|e^{-t\Lambda^\varepsilon}v|\rangle<\infty.
\]

Therefore, taking $v=\varphi^{-1}h$ we arrive at $(\star\star)$. Finally, it is seen that $\varphi e^{-t\Lambda^\varepsilon}\varphi^{-1}$ preserves positivity, so $(\bullet)$ follows from $(\star\star)$ by noticing that $e^{-t\Lambda^\varepsilon}|g|\uparrow e^{-t\Lambda}|g|$ $\mathcal L^d$ a.e.

\medskip

Let us write down a simple consequence of $(\star\star)$:

\begin{corollary}
\label{AHcorol1} For all $t>0$, $x\in\mathbb R^d-\{0\}$ and all small $\varepsilon>0$, there is a constant $\hat{c}$, such that
\[
e^{-t\Lambda^\varepsilon}\varphi_t\leq e^{\hat{c}}\varphi_t \text{ and } \langle e^{-t\Lambda^\varepsilon}(x,\cdot)\rangle\leq 2e^{\hat{c}}\varphi_t(x).
\]
\end{corollary}

\begin{proof}[Proof of Proposition \ref{prop_main}] First we note that, for $f=\phi_n u\in M$,
\begin{align*}
\langle Qf,\frac{f}{|f|}\rangle=&\langle \phi_n\Lambda^\varepsilon u,\frac{f}{|f|}\rangle=\lim_{t\downarrow 0}t^{-1}\langle\phi_n(1-e^{-t\Lambda^\varepsilon})u,\frac{f}{|f|}\rangle,\\
\Real\langle Qf,\frac{f}{|f|}\rangle &\geq\lim_{t\downarrow 0}t^{-1}\langle(1-e^{-t\Lambda^\varepsilon})|u|,\phi_n\rangle\\
&=\langle \Lambda^\varepsilon e^{-\frac{\Lambda^\varepsilon}{n}}|u|,\varphi\rangle.
\end{align*}
Let us emphasize that $e^{-t\Lambda^\varepsilon}$ is a holomorphic semigroup due to the Hille Perturbation Theorem (see e.g. \cite[Ch.\,IX, sect.\,2.2]{Ka}).

We are going to estimate $J:=\langle \Lambda^\varepsilon e^{-\frac{\Lambda^\varepsilon}{n}}|u|,\varphi\rangle$ ($=\langle  e^{-\frac{\Lambda^\varepsilon}{2n}}|u|,\Lambda^\varepsilon e^{-\frac{\Lambda^\varepsilon}{2n}}\varphi\rangle$)
from below using the equality
\[
(-\Delta)^\frac{\alpha}{2}\varphi=-I_{2-\alpha}\Delta\varphi, 
\]
where $I_\nu \equiv(-\Delta)^{-\frac{\nu}{2}}$.

Since $e^{-t\Lambda^\varepsilon}$ is a $C_0$ semigroup in $L^1$ and $C_{u}$, and $\varphi=\varphi_{(1)} + \varphi_{(u)}$,  $\varphi_{(1)} \in D((-\Delta)^{\frac{\alpha}{2}}_1)$, $\varphi_{(u)} \in D((-\Delta)^{\frac{\alpha}{2}}_{C_{u}})$,
$\Lambda^\varepsilon\varphi$ is well defined and belongs to $L^1 + C_u=\{w+v\mid w\in L^1, v\in C_u\}$.

Using the equality $(-\Delta)^\frac{\alpha}{2}\tilde{\varphi}_1=V\tilde{\varphi}_1$, where $\tilde{\varphi}_1(x)=|x|^{-d+\beta}$ (see e.g.\,\cite{KPS}), we have
\[
(-\Delta)^\frac{\alpha}{2}\varphi_1= V\tilde{\varphi}_1-I_{2-\alpha}\Delta(\varphi_1-\tilde{\varphi}_1)= V\tilde{\varphi}_1-I_{2-\alpha}\mathbf 1_{B^c(0,1)}\Delta(\varphi_1-\tilde{\varphi}_1).\quad B^c(0,1):=\mathbb R^d-B(0,1).
\]
Routine calculation shows that $-I_{2-\alpha}(\mathbf 1_{B^c(0,1)}\Delta(\varphi_1-\tilde{\varphi}_1)\geq -C_1$ for a constant $C_1$.

Since $\Lambda^\varepsilon\varphi_1=(-\Delta)^\frac{\alpha}{2}\varphi_1-V_\varepsilon\varphi_1$ and $V\tilde{\varphi}_1-V_\varepsilon\varphi_1\geq -V_\varepsilon(\varphi_1-\tilde{\varphi}_1)\geq -\delta c_\alpha^{-2}$, we obtain by scaling the bound
\[
J=\langle e^{-\frac{\Lambda^\varepsilon}{n}}|u|,\Lambda^\varepsilon \varphi\rangle \geq -(\delta c_\alpha^{-2} +C_1) s^{-1}\|e^{-\frac{\Lambda^\varepsilon}{n}}\|_{1\to 1}\|\phi_n^{-1}f\|_1, 
\]
or due to $\phi_n \geq \frac{1}{2}$,
\[
J\geq -2Cs^{-1}\|e^{-\frac{\Lambda^\varepsilon}{n}}\|_{1\to 1}\|f\|_1,  \quad C=C_1+\delta c_\alpha^{-2}.
\]

Noticing that $\|e^{-\frac{\Lambda^\varepsilon}{n}}\|_{1\to 1}\leq e^{\delta c_\alpha^{-2}\varepsilon^{-2} n^{-1}}=1+o(n)$ and taking $\lambda =3C s^{-1}$ we arrive at 
\[ 
 \Real\langle(\lambda+ Q)f,\frac{f}{|f|}\rangle\geq 0 \quad \quad f\in M. 
\] 
Clearly, the latter holds for all $f\in D(\tilde{Q})$. 
\end{proof}

The proof of \eqref{M3} is completed. We have verified all the assumptions of ($M_1$)-($M_4$) of Theorem A. The latter now yields the assertion of Theorem \ref{thm1}. 
\end{proof}

%

Having at hand Theorem \ref{thm1} and Corollary \ref{AHcorol1}, it is a simple matter to obtain the upper and lower bounds of the form 
\[
e^{-t\Lambda}(x,y)\approx e^{-t(-\Delta)^\frac{\alpha}{2}}(x,y)\varphi_t(x)\varphi_t(y). 
\]
Here $e^{-t(-\Delta)^\frac{\alpha}{2}}(x,y)\approx t^{-\frac{d}{\alpha}}\wedge\frac{t}{|x-y|^{d+\alpha}}$. ($a(z) \approx b(z)$ means that $c^{-1}b(z)\leq a(z) \leq cb(z)$ for some constant $c>1$ and all admissible $z$).

\medskip

\noindent\textbf{Proof of upper bound} $e^{-t\Lambda}(x,y)\leq C e^{-tA}(x,y)\varphi_t(x)\varphi_t(y)\quad (t>0,x,y\neq 0)$.  (For brevity here and below $(-\Delta)^\frac{\alpha}{2}=:A$.)

 By scaling, it suffices to consider $t=1$. Since $e^{-A}(x,y)\approx 1\wedge|x-y|^{-d-\alpha}$ $(x\neq y)$, Theorem \ref{thm1} yields, for $|x|, |y|\leq 2R$, 
\[
e^{-\Lambda^\varepsilon}(x,y)\leq C_R e^{-A}(x,y)\varphi(x)\varphi(y), \quad (\varphi\equiv\varphi_1) 
\]
By symmetry, it remains to prove this estimate for $|x|\leq |y|$, $|y|>2R$. First we note that for $|x|\leq |y|$, $|y|>2R$, $|z|\leq R$ and $0\leq\tau< 1$,
\[
e^{-(1-\tau)A}(z,y)\leq e^{-A}(x,y).
\]
Thus, by the Duhamel formula $e^{-\Lambda^\varepsilon}=e^{-A}+\int_0^1 e^{-\tau\Lambda^\varepsilon}V_\varepsilon e^{-(1-\tau)A} d\tau$,
\begin{align*}
e^{-\Lambda^\varepsilon}(x,y)&\leq e^{-A}(x,y)\bigg(1+\int_0^1 e^{-\tau \Lambda^\varepsilon} V_\varepsilon(x) d\tau\bigg)+\int_0^1 \langle e^{-\tau \Lambda^\varepsilon} (x,z)V_\varepsilon(z)\mathbf 1_{B^c(0,R)}(z)e^{-(1-\tau)A}(z,y)\rangle_z d\tau\\
&\leq e^{-A}(x,y)\bigg(1+\int_0^1 e^{-\tau \Lambda^\varepsilon} V_\varepsilon(x) d\tau\bigg)+V(R)\int_0^1 \langle e^{-\tau \Lambda^\varepsilon} (x,z)e^{-(1-\tau)A}(z,y)\rangle_z d\tau.
\end{align*}
Now fix $R$ by $\delta c_\alpha^{-2}R^{-\alpha}=\frac{1}{2}$. Then
\[
V(R)\int_0^1 \langle e^{-\tau \Lambda^\varepsilon} (x,z)e^{-(1-\tau)A}(z,y)\rangle_z d\tau\leq\frac{1}{2}\int_0^1 \langle e^{-\tau \Lambda^\varepsilon} (x,z)e^{-(1-\tau)\Lambda^\varepsilon}(z,y)\rangle_z d\tau=\frac{1}{2}e^{-\Lambda^\varepsilon}(x,y),
\]
and so
\[
\frac{1}{2}e^{-\Lambda^\varepsilon}(x,y)\leq e^{-A}(x,y)\bigg(1+\int_0^1 e^{-\tau \Lambda^\varepsilon}V_\varepsilon(x) d\tau\bigg).
\]
Next, by the Duhamel formula and Corollary \ref{AHcorol1},
\[
1+\int_0^1 e^{-\tau \Lambda^\varepsilon}V_\varepsilon(x) d\tau=\langle e^{-\Lambda^\varepsilon}(x,\cdot)\rangle\leq 2e^{\hat{c}}\varphi(x),
\]
and hence $e^{-\Lambda^\varepsilon}(x,y)\leq 4e^{\hat{c}}e^{-A}(x,y)\varphi(x)\leq 8e^{\hat{c}}e^{-A}(x,y)\varphi(x)\varphi(y)$. 

Finally, setting $C=C_R\vee (8e^{\hat{c}})$ and using $e^{-\Lambda^\varepsilon}|f|\uparrow e^{-\Lambda}|f|$ we end the proof of the upper bound.\hfill \qed
 
\medskip
\bigskip

\noindent\textbf{Proof of lower bound} $e^{-t\Lambda}(x,y)\geq C e^{-tA}(x,y)\varphi_t(x)\varphi_t(y)$ $\;\;(C>0$, $x,y\neq 0)$. 

\begin{proposition}
\label{prop_l}
Define $g=\varphi h$, $\varphi\equiv \varphi_s$, $0\leq h\in \mathcal S$-the L.Schwartz space of test functions. There is a constant $\hat{\mu}>0$ such that, for all  $0<t\leq s$,
\[
e^{-\frac{\hat{\mu}}{s}t}\langle g\rangle\leq \langle\varphi e^{-t\Lambda}\varphi^{-1} g\rangle.
\]
\end{proposition}
\begin{proof} [Proof of Proposition \ref{prop_l}]
Set $g_n=\phi_n h$, $\phi_n(x)=e^{-\frac{\Lambda^\varepsilon}{n}}\varphi(x)$, $\varphi\equiv\varphi_s$. Let $\mu>0$ be a constant. Then $(\mu=\frac{\hat \mu}{s})$
\[
\langle g_n\rangle-\langle \phi_n e^{-t(\Lambda^\varepsilon-\mu)}h\rangle=-\mu\int_0^t\langle\varphi,e^{-\tau(\Lambda^\varepsilon-\mu)}e^{-\frac{\Lambda^\varepsilon}{n}}h\rangle d\tau + \int_0^t\langle\varphi,\Lambda^\varepsilon e^{-\tau(\Lambda^\varepsilon-\mu)}e^{-\frac{\Lambda^\varepsilon}{n}}h\rangle d\tau.
\]
Note that $\Lambda^\varepsilon\varphi=\Lambda^\varepsilon\tilde{\varphi}+\Lambda^\varepsilon(\varphi-\tilde{\varphi})=\mathbf 1_{B(0,1)}(V-V_\varepsilon)\varphi+v_\varepsilon$, where $\tilde{\varphi}(x)=(s^{-\frac{1}{\alpha}}|x|)^{-d+\beta}$.
 Routine calculation shows that $\|v_\varepsilon\|_\infty\leq\frac{\mu_1}{s}$, $\;\;\mu_1\neq\mu_1(\varepsilon)$. Thus
 \[
\int_0^t\langle v_\varepsilon, e^{-\tau(\Lambda^\varepsilon-\mu)}e^{-\frac{\Lambda^\varepsilon}{n}}h\rangle d\tau\leq \frac{\mu_1}{s}\int_0^t\langle e^{-\tau(\Lambda^\varepsilon-\mu)}e^{-\frac{\Lambda^\varepsilon}{n}}h\rangle d\tau\leq \frac{2\mu_1}{s}\int_0^t\langle\varphi, e^{-\tau(\Lambda^\varepsilon-\mu)}e^{-\frac{\Lambda^\varepsilon}{n}}h\rangle d\tau.
\]
Taking $\hat{\mu}=2\mu_1$, we have
 \[
\langle g_n\rangle-\langle \phi_n e^{-t(\Lambda^\varepsilon-\mu)}h\rangle\leq \int_0^t\langle \mathbf 1_{B(0,1)}(V-V_\varepsilon)\varphi,e^{-(\tau+\frac{1}{n})\Lambda^\varepsilon}h\rangle e^{\mu \tau} d\tau, \text{ or sending } n\to\infty,
\] 
\[
\langle g\rangle-e^{\frac{\hat{\mu}}{s}t}\langle \varphi e^{-t\Lambda^\varepsilon}h\rangle\leq e^{\hat{\mu}}\int_0^t\langle \mathbf 1_{B(0,1)}(V-V_\varepsilon)\varphi,e^{-\tau\Lambda^\varepsilon}h\rangle d\tau.
\]
Set $W_\varepsilon= \mathbf{1}_{B(0,1)}(V-V_\varepsilon)\varphi^2$ and $F^\tau_\varepsilon=\varphi e^{-\tau\Lambda^\varepsilon}\varphi^{-1}$. Note that $W_\varepsilon\in L^1$ due to $2(d-\beta)+\alpha<d$, and $\|F^\tau_\varepsilon f\|_1\leq e^{\frac{\hat{c}}{s}\tau}\|f\|_1$, $f\in L^1$ due to Proposition \ref{prop_main}. Therefore,
\[
\int_0^t\langle \mathbf 1_{B(0,1)}(V-V_\varepsilon)\varphi,e^{-\tau\Lambda^\varepsilon}h\rangle d\tau = \int_0^t\langle F^\tau_\varepsilon W_\varepsilon,\varphi^{-1}h\rangle\leq 2e^{\hat{c}}s\|W_\varepsilon\|_1\|h\|_\infty\rightarrow 0 \text{ as } \varepsilon\downarrow 0.
\]
\end{proof}

We also need the following consequence of the upper bound and Proposition \ref{prop_l}.   

\begin{corollary}
\label{ANcorol2} 
Fix $t>0$. Set $g:=\varphi h$, $\varphi=\varphi_t$, $0\leq h\in \mathcal S$ with $\sprt h\in B(0,R_0)$ for some $R_0<\infty$. Then there are $0<r_t<R_0\vee t^\frac{\alpha}{2}<R_{t,R_0}$ such that, for all $r\in [0,r_t]$ and $R\in [2R_{t,R_0},\infty[$,
\[
e^{-\hat{\mu}-1}\langle g\rangle\leq\langle \mathbf 1_{R,r}\varphi e^{-t\Lambda}\varphi^{-1}g\rangle, \quad\quad \mathbf 1_{R,r}:=\mathbf 1_{B(0,R)}-\mathbf 1_{B(0,r)}, \;\; \mathbf 1_{R,0}:=\mathbf 1_{B(0,R)}.
\]
In particular, $e^{-\hat{\mu}-1}\varphi_t(x)\leq e^{-t\Lambda}\varphi_t\mathbf 1_{R,r}(x)$ for every $x \in B(0,R_0)$.
\end{corollary}
\begin{proof} [Proof of Corollary \ref{ANcorol2}]
By the upper bound,
\begin{align*}
\langle \mathbf 1_{B(0,r)}\varphi e^{-t\Lambda}\varphi^{-1} g\rangle & \leq C\langle \mathbf 1_{B(0,r)}\varphi^2,e^{-tA}g\rangle\\
&\leq CC_1t^{-\frac{d}{\alpha}}\|\mathbf 1_{B(0,r)}\varphi^2\|_1 \|g\|_1\\
& = o(r_t)\|g\|_1, \quad o(r_t)\rightarrow 0 \text{ as } r_t\downarrow 0;\\
\langle\mathbf 1_{B^c(0,R)}\varphi e^{-t\Lambda}\varphi^{-1} g\rangle & \leq C\langle\mathbf 1_{B^c(0,R)}\varphi^2, e^{-tA}g\rangle\\
& \leq C\langle e^{-tA}\mathbf 1_{B^c(0,R)},g\mathbf 1_{B(0,R_0)}\rangle, \text{ where } R\geq 2R_{t,R_0}\geq 2(R_0\vee t^\frac{\alpha}{2})\\
&\leq C\sup_{x\in B(0,R_0)} e^{-tA}\mathbf 1_{B^c(0,R)}(x)\|g\|_1\\
&\leq C\tilde{C}C_d R_{t,R_0}^{-\frac{\alpha}{2}}\|g\|_1\\
&=o(R_{t,R_0})\|g\|_1, \quad o(R_{t,R_0})\rightarrow 0 \text{ as } R_{t,R_0}\uparrow\infty
\end{align*}
due to $e^{-tA}(x,y)\leq \tilde{C}(t|x-y|^{-d-\alpha}\wedge t^{-\frac{d}{\alpha}})\leq \tilde{C}2^{d+\frac{\alpha}{2}}|y|^{-d-\frac{\alpha}{2}} \text{ if } |x|\leq R_0 \text{ and } |y|\geq R$.

We are left to apply Proposition \ref{prop_l}.
\end{proof}

Now we are in position to apply the so-called $3q$ argument. Set $q_t(x,\cdot)=e^{-t\Lambda}(x,\cdot)\varphi_t^{-1}(x)\varphi_t^{-1}(\cdot)$.

($a$) Let $x,y\in B^c(0,1)$, $x\neq y$. Clearly,
\[
q_3(x,y)\geq \varphi_3^{-1}(x)\varphi_3^{-1}(y)e^{-3\Lambda}(x,y)\geq e^{-3\Lambda}(x,y)\geq e^{-3A}(x,y).
\] 

($b$) Let $x,y\in B(0,1)$, $0< |x|\leq |y|$. By the reproduction property, since $e^{-t\Lambda}$ is positivity preserving,
\begin{align*}
q_3(x,y)&\geq \varphi_3^{-1}(x)\varphi_3^{-1}(y)\langle e^{-\Lambda}(x,\cdot)e^{-2\Lambda}(\cdot,y)\mathbf 1_{R,r}(\cdot)\rangle\\
&=\varphi_3^{-1}(x)\varphi_3^{-1}(y)\langle e^{-\Lambda}(x,\cdot)\varphi_1(\cdot)\varphi_1^{-1}(\cdot) e^{-2\Lambda}(\cdot,y)\mathbf 1_{R,r}(\cdot)\rangle \\
&\geq \varphi_3^{-1}(x)\varphi_3^{-1}(y)\langle e^{-\Lambda}(x,\cdot)\varphi_1(\cdot)\mathbf 1_{R,r}(\cdot)\rangle \inf_{r\leq |z|\leq R}\varphi_1^{-1}(z) e^{-2\Lambda}(z,y)\\
&\text{(here we are using Corollary \ref{ANcorol2})}\\
&\geq e^{-\hat{\mu}-1}\varphi_3^{-1}(x)\varphi_1(x)\varphi_1^{-1}(r)\varphi_3^{-1}(y)\inf_{r\leq |z|\leq R} e^{-2\Lambda}(z,y)\\
&=C_{r,R}\varphi_3^{-1}(y)\inf_{r\leq |z|\leq R} e^{-2\Lambda}(y,z);\\
e^{-2\Lambda}(y,z)&\geq \langle e^{-\Lambda}(y,\cdot)\varphi_1(\cdot)\varphi_1^{-1}(\cdot)e^{-\Lambda}(\cdot,z)\mathbf 1_{R,r}(\cdot)\rangle \\
&\text{(again we are using Corollary \ref{ANcorol2})}\\
&\geq e^{-\hat{\mu}-1}\varphi_1(y)\varphi_1^{-1}(r)\inf_{r\leq |z|,|\cdot|\leq R}e^{-\Lambda}(\cdot,z).\\
\end{align*}
Therefore
\[
q_3(x,y)\geq C_{r,R}^\prime \inf_{r\leq |z|,|\cdot|\leq R}e^{-A}(\cdot,z)\geq C_{r,R}^{\prime\prime}e^{-3A}(x,y).
\]

($c$) Let $x\in B(0,1)$, $x\neq 0$, $y\in B^c(0,1)$. Then 
\begin{align*}
q_3(x,y)&\geq \varphi_3^{-1}(x)\varphi_3^{-1}(y)\langle e^{-\Lambda}(x,\cdot)\varphi_1(\cdot) \varphi_1^{-1}(\cdot)e^{-2A}(\cdot,y)\mathbf 1_{R,r}(\cdot)\rangle\\
&\geq \varphi_1^{-1}(x)\langle e^{-\Lambda}(x,\cdot)\varphi_1(\cdot) \varphi_1^{-1}(\cdot)e^{-2A}(\cdot,y)\mathbf 1_{R,r}(\cdot)\rangle\\
&\geq e^{-\hat{\mu}-1}\inf_{r<|z|<R}\varphi_1^{-1}(z)e^{-2A}(z,y)\geq e^{-\hat{\mu}-1}\varphi_1^{-1}(r)\inf_{r<|z|<R}e^{-2A}(z,y)\\
&\geq C_{R,r} e^{-3A}(x,y).
\end{align*}

Finally, by ($a$),($b$),($c$), $q_3(x,y)\geq Ce^{-3A}(x,y)$ or $e^{-3\Lambda}(x,y)\geq Ce^{-3A}(x,y)\varphi_3(x)\varphi_3(y)$. The scaling argument ends the proof of the lower bound.\qed

\medskip
\noindent\textbf{3.~}The case $\delta=1$.
The following construction is standard: Define $\Lambda$ as the (minus) generator of a $C_0$ semigroup
$$
U^t:=s\mbox{-}L^2\mbox{-}\lim_{\varepsilon \downarrow 0}e^{-t\Lambda^{\varepsilon}}.
$$
(Indeed, set $u_\varepsilon(t)=e^{-t\Lambda^\varepsilon} f$, $f\in L^2_+.$ Since $\{V_\varepsilon\}$ is motonoically increasing as $\varepsilon \downarrow 0$, $u_\varepsilon \uparrow u$ to some $u$. Since $\|u_\varepsilon\|_2 \leq \|f\|_2$, we have
\[
u \in L^2_+, \quad\|u\|_2 \leq \|f\|_2, \quad \|u_\varepsilon \|_2 \uparrow \|u\|_2, \qquad  u_\varepsilon \overset{s}\rightarrow u=: U^t f, \quad U^{t+s}=U^tU^s, \quad U^tf\overset{s}\rightarrow f \text{ as } t\downarrow 0.
\]
For $f \in \Real L^2$ set $U^t f:=U^t f_+ - U^t f_-, \; f_\pm =0 \vee (\pm f)$. Then $\|U^t f\|_2 \leq \|U^t|f|\|_2 = \lim_\varepsilon \|e^{-t\Lambda^\varepsilon} |f|\|_2 = \|U^t |f|\|_2 \leq \|f\|_2.$)

By the fractional variant of the Brezis-Vasquez inequality \cite{BV}, see \cite{FLS}, 
$$
\langle \Lambda u, u \rangle=\|(-\Delta)^{\frac{\alpha}{4}}u\|_2^2-c_\alpha^{-2}\||x|^{-\frac{\alpha}{2}}u\|_2^2 \geq C_d\|u\|_{2j}^2,  \quad j \in [1,\frac{d}{d-\alpha}[,
$$  
Theorem A yields the on-diagonal bound $e^{-t\Lambda}(x,y) \leq Ct^{-j'}\varphi_t(x)\varphi_t(y)$ with $j'>\frac{d}{\alpha}$. 
In particular, $e^{-t\Lambda}(x,y) \leq C\varphi_1(x)\varphi_1(y)$.
Now, repeating the argument in the previous section, we obtain the upper bound
$$
e^{-\Lambda}(x,y) \leq e^{-(-\Delta)^{\frac{\alpha}{2}}}(x,y)\varphi_1(x)\varphi_1(y),
$$
and so, by scaling, 
$$
e^{-t\Lambda}(x,y) \leq e^{-t(-\Delta)^{\frac{\alpha}{2}}}(x,y)\varphi_t(x)\varphi_t(y), \quad t>0.
$$
The proof of the lower bound remains the same.

\begin{remark}
The observation that the scaling properties of $e^{-t(-\Delta)^{\frac{\alpha}{2}}}$, $0<\alpha<2$, 
allow to obtain the optimal upper bound even when the Sobolev embedding of $\Lambda$ is valid only for $j<\frac{d}{d-\alpha}$ is due to \cite{BGJP}. 
This is in sharp contrast to the case $\alpha=2$. Indeed, the scaling properties of $e^{t\Delta}$ are different, so one needs another argument (i.e.\,to pass to a space of higher dimension where one can appeal to the V.\,P.\,Il'in-Sobolev inequality \cite{MS}).
\end{remark}

\appendix

\section{Proof of Theorem A}

\label{A}

Set $L^2_\varphi =L^2(X,\varphi^2 d\mu)$, and define a unitary map $\Phi: L^2_\varphi \to L^2$  by $\Phi f=\varphi f$. Then the operator $\Lambda_\varphi = \Phi^{-1}\Lambda\Phi$ of domain $D(\Lambda_\varphi)=\Phi^{-1}D(\Lambda)$ is selfadjont on $L^2_\varphi$ and $\|e^{-t\Lambda_\varphi}\|_{2 \to 2, \varphi} =\|e^{-t\Lambda}\|_{2\to 2} \leq 1$ for all $t\geq 0$. Here and below the subscript $\varphi$ indicates that the corresponding quantities are related to the measure $\varphi^2 d \mu$.

Let $f=\varphi^{-1}h$, $h\in L^\infty_{com}$, and so $f\in L^2_\varphi\cap L^1_\varphi$ by $(M_2)$. Let $u_t=e^{-t \Lambda_\varphi}f$. Then $\varphi u_t = e^{-t\Lambda} \varphi f$ and
\begin{align*}
\langle \Lambda_\varphi u_t,u_t\rangle_\varphi & = \|\Lambda^\frac{1}{2}\varphi u_t\|_2^2 \geq c_S\|\varphi u_t\|_{2j}^2\\
& \geq c_S\|\varphi u_t\|_2^{2+\frac{2}{j^\prime}} \|\varphi u_t\|_1^{-\frac{2}{j^\prime}}\\
& = c_S\langle u_t,u_t \rangle_\varphi^{1+\frac{1}{j^\prime}} \|\varphi^{-1} \varphi e^{-t\Lambda} \varphi^{-1} \varphi^2f\|_1^{-\frac{2}{j^\prime}},
\end{align*}
where $(M_1)$ and H\"older's inequality have been used.

Clearly, $-\frac{1}{2}\frac{d}{d t}\langle u_t,u_t\rangle_\varphi = \langle \Lambda_\varphi u_t,u_t\rangle_\varphi$. Setting $w:=\langle u_t,u_t\rangle_\varphi$ and using $(M_4)$ we have
\[
\frac{d}{d t} w^{-\frac{1}{j^\prime}} \geq \frac{2}{j^\prime}c_S(c_0^{-1} \|\varphi e^{-t\Lambda} \varphi^{-1} \varphi^2 f\|_1)^{-\frac{2}{j^\prime}}.
\]
By our choice of $f$, $\varphi^2f =\varphi h \in \mathcal D$. Therefore we can apply $(M_3)$ and obtain
\[
\frac{d}{d t} w^{-\frac{1}{j^\prime}}\geq \frac{2}{j^\prime}c_S(c_1c_0^{-1} \|f\|_{1,\varphi})^{-\frac{2}{j^\prime}}, \;\; t\leq s.
\]
Integrating this inequality over $[0,t]$ gives
\[
\|e^{-t\Lambda_{\varphi_s}}f\|_{2,\varphi_s} \leq c t^{-\frac{j^\prime}{2}} \|f\|_{1,\varphi_s}, \;\; t\leq s. 
\]
Since $f\in \varphi^{-1} L^\infty_{com}$ and $\varphi^{-1} L^\infty_{com}$ is a dense subspace of $L^1_\varphi$, the last inequality yields
\[
\|e^{-t\Lambda_{\varphi_s}} \|_{1\to 2,\varphi_s} \leq c t^{-\frac{j^\prime}{2}}, \;\; t\leq s
\]
and \eqref{nie} follows. \hfill \qed

\end{document}